\documentclass[12pt]{amsart}
\usepackage{amssymb, amscd, amsmath, amsthm, latexsym, enumerate}
\newtheorem{theorem}{Theorem}
\newtheorem{lemma}[theorem]{Lemma}

\begin{document}

\title{$\mathbb{S}ol^3\times\mathbb{E}^1$-manifolds}

\author{J. A. Hillman }
\address
{{School of Mathematics and Statistics, University of Sydney,}
\newline
Sydney,  NSW 2006, Australia }

\email{jonathan.hillman@sydney.edu.au}



\begin{abstract}
We show that closed $\mathbb{S}ol^3\times\mathbb{E}^1$-manifolds
are Seifert fibred, with general fibre the torus, and base 
one of the flat 2-orbifolds $T, Kb,\mathbb{A},\mathbb{M}b,
S(2,2,2,2), P(2,2)$ or $\mathbb{D}(2,2)$.
\end{abstract}

\maketitle

A closed 4-manifold $M$ is homeomorphic to an infrasolvmanifold if and only if 
$\chi(M)=0$ and $\pi_1(M)$ is torsion free and virtually poly-$Z$, of Hirsch length 4.
Every such group is realised in this way, 
and $M$ is determined up to homeomorphism by $\pi$.
Such manifolds are either mapping tori of self-homeomorphisms of 3-dimensional infrasolvmanifolds or are unions of two twisted $I$-bundles over such 3-manifolds.
(See Chapter 8 of \cite{Hi}.)

There are six families of 4-dimensional infrasolvmanifolds,
corresponding to the geometries $\mathbb{E}^4$, 
$\mathbb{N}il^3\times\mathbb{E}^1$, $\mathbb{N}il^4$, $\mathbb{S}ol_0^4$, 
$\mathbb{S}ol_1^4$ and $\mathbb{S}ol_{m,n}^4$ of solvable Lie type.
The 74 flat 4-manifolds can be listed,
while $\mathbb{N}il^3\times\mathbb{E}^1$- and $\mathbb{N}il^4$-manifolds
(infranilmanifolds of dimension 4) were classified in \cite{De}.
Every torsion-free, virtually poly-$Z$ group of Hirsch length 4 which is not virtually nilpotent
is the fundamental group of a  4-manifold with one of the remaining geometries \cite{Hi07}.
Manifolds with geometry $\mathbb{S}ol_{m,n}^4$ (with $m\not=n$) 
or $\mathbb{S}ol_0^4$ are mapping tori of self-homeomorphisms
of the 3-torus $\mathbb{R}^3/\mathbb{Z}^3$, 
and so may be classified in terms of conjugacy classes of matrices in $GL(3,\mathbb{Z})$.
The relationship between the various $\mathbb{S}ol_{m,n}^4$ geometries 
is not obvious (see page 137 of \cite{Hi}.)
However, when $m=n$ all agree with the product geometry 
$\mathbb{S}ol^3\times\mathbb{E}^1$.
Partial classifications of $\mathbb{S}ol^3\times\mathbb{E}^1$- 
and $\mathbb{S}ol_1^4$-manifolds were given in \cite{Co}.
A complete classification of $\mathbb{S}ol_1^4$-manifolds has recently appeared  \cite{Th12}.

The first section is on notation and terminology,
and \S2 gives some simple observations on subgroups of $GL(2,\mathbb{Z})$ 
with two ends.
In \S3 we show that $\mathbb{S}ol^3\times\mathbb{E}^1$-manifolds 
have canonical Seifert fibrations, with general fibre the torus, 
and base one of the seven flat 2-orbifolds $T, Kb,\mathbb{A},\mathbb{M}b,
S(2,2,2,2), P(2,2)$ or $\mathbb{D}(2,2)$.
The fibration is unique, and so
this suggest a route to the classification of such manifolds, 
in which the key elements are the base $B$,
the action $\alpha$ of $\pi_1^{orb}(B)$ on $N=\pi_1(F)$,
where $F$ is the general fibre, 
and an ``Euler class" in $H^2(\beta;N^\alpha)$.
The manifolds which are mapping tori may also be classified in terms 
of conjugacy classes of automorphisms.
In \S4 we consider the interaction of the Seifert fibrations, 
mapping torus structure and orientability for such manifolds,
but shall not otherwise classify them explicitly.
The others all have base orbifold either $S(2,2,2,2)$,
$P(2,2)$ or $\mathbb{D}(2,2)$.
The orbifold fundamental groups all admit epimorphisms to 
the infinite dihedral group $Z/2Z*Z/2Z$, and so the 
$\mathbb{S}ol^3\times\mathbb{E}^1$-group has a corresponding decomposition 
as a generalized free product with amalgamation.
We use this structure in \S5 to give examples of each of these three types.
In \S6 we sketch why the Seifert approach should suffice to classify
the $\mathbb{S}ol^3\times\mathbb{E}^1$-manifolds with such bases,
but we do not pursue such a classification in detail.

\section{notation and terminology}

If $G$ is a group let $G'$ be its commutator subgroup and let 
$\beta_1(G)$ be the rank of $G/G'$.
Let $I(G)$ be the preimage of the torsion of $G/G'$ in $G$,
and let $\sqrt{G}$ be the Hirsch-Plotkin radical of $G$.
(In the cases of interest below, 
$\sqrt{G}$ is the unique maximal nilpotent normal subgroup of $G$.)
Let $X^2(G)$ be the subgroup generated by squares.
If $x\in{G}$ let $c_x$ be the automorphism induced by conjugation by $x$.
If $H\leq{G}$ let $C_G(H)$ and $N_G(H)$ be the centralizer 
and normalizer of $H$ in $G$, respectively.
In particular, $\zeta{G}=C_G(G)$ is the centre of $G$.
If $G$ is virtually solvable let $h(G)$ be its Hirsch length.

The symbols $G_1,\dots,G_6$  and $B_1,\dots,B_4$ denote the six orientable 
and four non-orientable flat 3-manifold groups, respectively.
(See Chapter 8 of \cite{Hi}.)

Our notation for flat 2-orbifolds is taken from Appendix A of \cite{Mo},
embellished with ``blackboard bold" font for the initial letters 
of the symbols for such orbifolds with reflector curves.
Similarly, $\mathbb{I}$ denotes the reflector interval
(the quotient of $S^1$ by complex conjugation).
(This font is otherwise used for the integers, rationals and real numbers, 
and for the initial letters of names of geometries.
We use italics for the symbols for the associated model spaces, 
in this paper the Lie group $Sol^3\times\mathbb{R}$.)

\section{some lemmas on subgroups of $GL(2,\mathbb{Z})$}

Let $D_\infty=Z/2Z*Z/2Z$ be the infinite dihedral group,
with presentation $\langle{u,v}\mid{u^2=v^2=1}\rangle$.
Recall that a group $G$ has two ends  $\Leftrightarrow$
$G$ has an infinite cyclic subgroup of finite index $\Leftrightarrow$
$G$ has a maximal finite normal subgroup with quotient $\mathbb{Z}$ or $D_\infty$.

\begin{lemma}
Let $F$ be a finite subgroup of $G=GL(2,\mathbb{Z})$.
If $N_G(F)$ is infinite then $F\leq\{\pm{I}\}$.
\end{lemma}

\begin{proof}
If $P\in{F}\setminus\{\pm{I}\}$ then it is conjugate to one of 
 $\left(\begin{smallmatrix}
1&0\\ 
0&-1
\end{smallmatrix}\right)$,
 $\left(\begin{smallmatrix}
0&1\\ 
1&0
\end{smallmatrix}\right)$,
 $\left(\begin{smallmatrix}
0&1\\ 
-1&-1
\end{smallmatrix}\right)$,
 $\left(\begin{smallmatrix}
0&1\\ 
-1&0
\end{smallmatrix}\right)$ or
 $\left(\begin{smallmatrix}
0&1\\ 
-1&1
\end{smallmatrix}\right)$.
(These have orders 2, 2, 3, 4 and 6, respectively.)
In each case $C_G(P)$ is finite, and so $C_G(F)$ is finite.
Since $Aut(F)$ is finite the lemma follows.
\end{proof}

Since we may assume without loss of generality that $-I\in{F}$,
this lemma also follows from the fact that 
$PSL(2,\mathbb{Z})\cong{Z/2Z*Z/3Z}$.

\begin{lemma}
Let $H<GL(2,\mathbb{Z})$ have two ends.
Then either 
\begin{enumerate}
\item$H\cong\mathbb{Z}$; or
\item$H\cong\mathbb{Z}\oplus\langle-I\rangle$; or
\item $H=\langle{A,B}\rangle$ where $A^2=B^2=I$; or
\item $H=\langle{A,B,-I}\rangle$ where $A^2=B^2=I$; or
\item $H=\langle{A,B}\rangle$ where $A^2=-I$ and $B^2=I$; or 
\item $H=\langle{A,B}\rangle$ where $A^2=B^2=-I$.
\end{enumerate}
In each case neither $A$ nor $B$ is $-I$.
\end{lemma}

\begin{proof}
Let $F$ be the maximal finite normal subgroup of $H$.
Then $F\leq\{\pm{I}\}$, by Lemma 1.
If $H/F\cong\mathbb{Z}$ then (1) or (2) holds.

Suppose that $H/F\cong{D_\infty}$, and 
let $A, B\in{H}$ represent generators of the free factors of $D_\infty$.
Then $A$ and $B$ each have order dividing 4.
Since $AB$ has infinite order, neither $A$ nor $B$ is $-I$.
If $A$ has order 2 then $\det(A)=-1$,
while if $A$ has order 4 then $\det(A)=+1$ and $A^2=-I$,
and similarly for $B$.
Thus if $F=1$ then (3) holds,
while if $F=\{\pm{I}\}$ then (4), (5) or (6) holds.
\end{proof}

Let $\widetilde{D}_\infty=\langle{a,b}\mid{a^2=b^4=1},~ab^2=b^2a\rangle$
be the central extension of $D_\infty$ arising in case (5) of Lemma 2.

\section{seifert fibrations}

Let $\mathcal{S}=Sol^3\times\mathbb{R}$.
Then $\mathcal{S}\cong\mathbb{R}^3\rtimes_\Theta\mathbb{R}$,
where $\Theta(t)=diag[e^t,1,e^{-t}]$,
and $\mathcal{S}$ is the identity component of the isometry group $Isom(\mathbb{S}ol^3\times\mathbb{E}^1)$.
(Its group of path components is ${D_8\times{Z/2Z}}$). 
The nilradical of $\mathcal{S}$ is $\sqrt{\mathcal{S}}\cong\mathbb{R}^3$,
and $\mathcal{S}/\sqrt\mathcal{S}\cong\mathbb{R}$.
(See \cite{Wl}, or page 137 of \cite{Hi}).

\begin{theorem}
Let $M$ be a $\mathbb{S}ol^3\times\mathbb{E}^1$-manifold.
Then $M$ has an essentially unique Seifert fibration, with general fibre $T$
and base $B= T$, $Kb$, $\mathbb{A}$, $\mathbb{M}b$,  $S(2,2,2,2)$, $P(2,2)$
or $\mathbb{D}(2,2)$.
In particular, $\beta_1(M)\leq2$.
\end{theorem}

\begin{proof}
The manifold $M$ is a quotient of the Lie group $\mathcal{S}$ by a lattice 
$\pi=\pi_1(M)$ in $Isom(\mathbb{S}ol^3\times\mathbb{E}^1)$.
The foliation of $\mathcal{S}$ by translates of its commutator subgroup
is preserved by the isometry group, 
and so it induces a canonical foliation on $M$.
The leaf map for this foliation is a Seifert fibration $p:M\to{B}$,
with base $B$ a flat 2-orbifold and general fibre $F$ a flat 2-manifold.
Hence $\pi$ has a normal subgroup $N$ such that 
$\pi/N\cong\beta=\pi_1^{orb}(B)$ is a flat 2-orbifold group.

Conjugation in $\pi$ determines an ``action" homomorphism $\alpha:\beta\to{Out}(N)$.
Since $\pi$ is not virtually nilpotent, $\mathrm{Im}(\alpha)$ is infinite.
Hence $F\cong{T}$, since $Out(\pi_1(Kb))$ is finite.
Therefore $N\cong\mathbb{Z}^2$, and so $Out(N)\cong{GL(2,\mathbb{Z})}$ is virtually free.
Since $\mathrm{Im}(\alpha)$ is an infinite solvable group, 
it must be virtually $\mathbb{Z}$.
Hence $B$ fibres over $S^1$ or $\mathbb{I}$.

Let $\widehat{M}$ be the finite covering space induced from a torus
$\widehat{B}$ covering $B$,
and let $\widehat\pi=\pi_1(\widehat{M})$.
Then $\beta_1(\widehat\pi)=\beta_1(\widehat\pi/N)=2$, 
and so $N$ and $\widehat\pi'$ are commensurable.
Hence $N$ has finite image in $\pi/\pi'$.
In particular, $\beta_1(\pi)=\beta_1(\beta)\leq2$.

Suppose that $q:M\to\overline{B}$ is another Seifert fibration and 
$\overline{N}$ is the fundamental group of the general fibre.
The base $\overline{B}$ is a flat 2-orbifold, since $\pi$ is solvable,
and again must itself fibre over $S^1$ or $\mathbb{I}$.
After passing to a subgroup $\widehat\pi$ of finite index, if necessary, 
we may assume that $\widehat\pi/N\cong\mathbb{Z}^2$.
Since $N$ and $\overline{N}\cap\widehat\pi$ have finite image in $\widehat\pi/\widehat\pi'$,
$N$ and $\overline{N}$ must each be commensurable with $\widehat\pi'\cong\mathbb{Z}^2$.
Thus $N$ and $\overline{N}$ each have finite index in $N\overline{N}$.
Since the groups of flat 2-orbifolds do not have non-trivial finite normal subgroups
it follows that $N=\overline{N}$.
Thus the fibration is unique (up to automorphisms of the base).

Let $\tilde\alpha$ be the composition of 
the projection of $\pi$ onto $\beta$ with the action $\alpha$,
and let $\nu=\tilde\alpha^{-1}(F)$,
where $F$ is the maximal finite normal subgroup of $\mathrm{Im}(\alpha)$.
Since $N<\sqrt\pi\cong\mathbb{Z}^3$, 
we see that $\sqrt\pi\leq\mathrm{Ker}(\tilde\alpha)$.
Since $N$ is central in $\mathrm{Ker}(\tilde\alpha)$, 
which is a torsion-free virtually poly-$Z$ group of Hirsch length 3,
it follows that $\mathrm{Ker}(\tilde\alpha)=\sqrt\pi$,
and $\nu=\sqrt\pi$ or $\nu\cong{G_2}$.
Since $\sqrt\pi/N$ is an abelian normal subgroup of $\beta=\pi/N$
and $\beta$ has no nontrivial finite normal subgroup $\sqrt\pi/N\cong\mathbb{Z}$,
and so $\sqrt\pi\cong{N}\times\mathbb{Z}$.

If $\nu=\sqrt\pi$ then $\nu/N\cong\mathbb{Z}$.
If $\nu\cong{G_2}$ then $N=I(\nu)$, and so we again have $\nu/N\cong\mathbb{Z}$.
In each case $B$ must fibre over $S^1$ or $\mathbb{I}$ with general fibre $S^1$,
and so $B=T$, $Kb$, $\mathbb{A}$, $\mathbb{M}b$,  $S(2,2,2,2)$, $P(2,2)$,
or $\mathbb{D}(2,2)$.
\end{proof}

The existence of such a Seifert fibration is discussed briefly on pages 146 and 176 of  \cite{Hi}.

The torus only fibres over $S^1$, the next three have fibrations 
of both kinds, while the remaining three only fibre over $\mathbb{I}$.
If $\pi/\pi'$ is finite then $B=S(2,2,2,2)$, $P(2,2)$ or $\mathbb{D}(2,2)$,
and in each case the epimorphism from $\beta$ to $\pi_1^{orb}(\mathbb{I})\cong{D_\infty}$ 
is unique up to composition with an automorphism of $\beta$.
(This is easily verified by considering the 
infinite cyclic normal subgroups of $\beta$.)
If $B=S(2,2,2,2)$ or $\mathbb{D}(2,2)$ there is also an essentially unique epimorphism to $D_\infty\times{Z/2Z}$,  
but none to $\widetilde{D}_\infty$.
If $B=P(2,2)$ there is no epimorphism to $D_\infty\times{Z/2Z}$,
but there is one to $\widetilde{D}_\infty$.
There are no actions with image as in case (6) of Lemma 2.

We shall show below that each of  these seven flat 2-orbifolds
is the base of the Seifert fibration of some
$\mathbb{S}ol^3\times\mathbb{E}^1$-manifold.
(In the first version of this paper it was erroneously claimed that 
$P(2,2)$ and all its quotients could be excluded.)
There are two further flat 2-orbifolds which fibre over 1-orbifolds,
namely $\mathbb{D}(2,\overline{2},\overline{2})$ and
$\mathbb{D}(\overline{2},\overline{2},\overline{2},\overline{2})$,
but these do not arise here, by Theorem 3.

We may derive some of the algebraic consequences of the theorem as follows.
The intersection $T(\pi)=\pi\cap\mathcal{S}$ is the
{\it translation subgroup\/} of $\pi$,
and $[\pi:T(\pi)]$ divides 16, 
since 
$Isom(\mathbb{S}ol^3\times\mathbb{E}^1)/\mathcal{S}\cong{D_8\times{Z/2Z}}$. 
Hence $h(T(\pi))=h(\pi)=4$.
Moreover, $\pi\cap\sqrt{\mathcal{S}}=T(\pi)\cap\sqrt{\mathcal{S}}\cong\mathbb{Z}^3$
\cite{Rg}.
Since this is normal in $\pi$, 
we see that $\mathbb{Z}^3\leq\sqrt\pi$.
Elements of $T(\pi)$ not in $\sqrt\mathcal{S}$ act on
$\pi\cap\sqrt{\mathcal{S}}$ via matrices with eigenvalues
$\{\xi,1,\xi^{-1}\}$, where $\xi>1$.
It follows easily that $h(\sqrt\pi)=3$.
Hence $\sqrt\pi\cong\mathbb{Z}^3$, 
since it is torsion-free nilpotent and 
$\pi\cap\sqrt{\mathcal{S}}\leq\sqrt\pi$.
Moreover, $\sqrt\pi\cap{T(\pi)'}=\sqrt\pi\cap\mathcal{S}'\cong\mathbb{Z}^2$.
Let $N$ be the preimage in $\pi$ of the maximal finite normal subgroup of $\pi/\sqrt\pi\cap\pi'$.
Then $N$ is a characteristic subgroup of $\pi$ and $N\cong\mathbb{Z}^2$.
The quotient $\pi/N$ is virtually $\mathbb{Z}^2$,
since $h(\pi/N)=2$,
and since it has no notrivial finite normal subgroup 
it is a flat 2-orbifold group.

Closed $\mathbb{N}il^3\times\mathbb{E}^1$- and $\mathbb{N}il^4$-manifolds 
also have canonical Seifert fibrations.
For these, the images of the fundamental group of the general fibre in $\pi$
are $\zeta\sqrt\pi$ and $\zeta_2\sqrt\pi$ (the second stage of the upper
central series), respectively.
In general, $\mathbb{N}il^3\times\mathbb{E}^1$-manifolds may have many Seifert fibrations,
but in the $\mathbb{N}il^4$ case the fibration is unique.

\section{mapping tori}

Let $\pi$ be the fundamental group of a $\mathbb{S}ol^3\times\mathbb{E}^1$-manifold $M$.
If $\pi/\pi'$ is infinite then $\pi$ is a semidirect product $\kappa\rtimes\mathbb{Z}$,
where $\kappa$ is a torsion-free virtually poly-$Z$ group of Hirsch length 3.
Either $\kappa$ is the group of a $\mathbb{S}ol^3$-manifold 
or $\sqrt\pi\leq\kappa$,
and then $[\kappa:\sqrt\pi]\leq2$,
by Theorem 8.4 of \cite{Hi}.
Such semidirect products may be classified 
in terms of conjugacy classes in $Out(\kappa)$.
In this section we shall consider the interactions between Seifert fibrations, 
mapping tori and orientability for these groups.
We shall consider the groups with $\pi/\pi'$ finite in later sections.

\begin{lemma}
If $\pi\cong\kappa\rtimes\mathbb{Z}$, 
where $\kappa=\mathbb{Z}^3$ or $G_2$ or is a $\mathbb{S}ol^3$-group 
such that $\kappa/\sqrt\kappa\cong\mathbb{Z}$, then $B=T$ or $Kb$.

Conversely, if $B=T$ or $Kb$ and $\sqrt\pi\leq\kappa$ then 
$\pi\cong\sqrt\pi\rtimes\mathbb{Z}$ or $G_2\rtimes\mathbb{Z}$.
\end{lemma}

\begin{proof}
In each case $N<\kappa$ and $\kappa/N\cong\mathbb{Z}$.
Hence $\beta\cong\mathbb{Z}^2$ or $\mathbb{Z}\rtimes_{-I}\mathbb{Z}$ 
and so $B=T$ or $Kb$.

If $B=T$ or $Kb$ and $\sqrt\pi\leq\kappa$ then $\mathrm{Im}(\alpha)$ 
is $\mathbb{Z}$ or $\mathbb{Z}\oplus\langle-I\rangle$, by Lemma 2,
and so $\pi\cong\sqrt\pi\rtimes\mathbb{Z}$ or $G_2\rtimes\mathbb{Z}$.
\end{proof}
 
In Theorem 3 it was shown that $\beta_1(\pi)=\beta_1(\beta)\leq2$,
and clearly $\beta_1(\pi)=2$ if and only if $B=T$.
If so, then $\pi\cong\sqrt\pi\rtimes\mathbb{Z}$ or $G_2\rtimes\mathbb{Z}$,
since $\mathrm{Im}(\alpha)$ is abelian.
The group $\pi$ is also a semidirect product $\sigma\rtimes\mathbb{Z}$,
where $\sigma$ is a $\mathbb{S}ol^3$-group (with $\sigma/\sqrt\sigma\cong\mathbb{Z}$),
in infinitely many ways.
(However, $\pi$ need not be a direct product $\sigma\times\mathbb{Z}$.)
There are orientable examples and non-orientable examples.
(All $T$-bundles over $T$ have been classified, in terms of extension data \cite{SF83}.
However Proposition 3 of \cite{SF83} appears to overlook some cases.)

If $\beta(\pi)=1$ then $B=Kb,\mathbb{A}$ or $\mathbb{M}b$, 
and the splitting of $\pi$ as a semidirect product is unique.
If $B=Kb$ there are orientable and non-orientable examples 
with $\pi\cong\sqrt\pi\rtimes\mathbb{Z}$ and with $\sigma\rtimes\mathbb{Z}$,
where $\sigma$ is a $\mathbb{S}ol^3$-group 
such that $\sigma/\sqrt\sigma\cong\mathbb{Z}$.
(See \S8 of Chapter 8 of \cite{Hi}.)
Conversely, if $\pi\cong\sigma\rtimes\mathbb{Z}$,
where $\sigma/\sqrt\sigma\cong\mathbb{Z}$ then $B=T$ or $Kb$.

\begin{lemma}
Let $\sigma$ be a $\mathbb{S}ol^3$-group such that $\sigma/\sqrt\sigma\cong{D_\infty}$.
Then $\sigma$ is orientable, and automorphisms of $\sigma$ are orientation preserving.
\end{lemma}

\begin{proof}
The hypotheses imply that $\sigma/\sigma'$ is finite.
Thus $H_1(\sigma;\mathbb{Q})=0$.
Since $\sigma$ is a $PD_3$-group, $\chi(\sigma)=0$,
Therefore $H_3(\sigma;\mathbb{Q})\not=0$, and so $\sigma$ is orientable.
(This can also be deduced from the fact that if $N\in{GL(2,\mathbb{C})}$ 
is conjugate to $N^{-1}$  then either $\det(N)=1$ or $N^2=1$.)
Let $[\sigma]\in{H_3}(\sigma;\mathbb{Z})$ be a generator.

Let $u$ and $v\in\sigma$ represent generating involutions of $D_\infty$, and let $t=uv$.
Let $f$ be an automorphism of $\sigma$.
Then $f$ restricts to an automorphism of $\sqrt\sigma$, 
and induces an automorphism  of $\sigma/\sqrt\sigma$.
After composition with an inner automorphism of $\sigma$,
if necessary, we may assume that either $f(u)\equiv{u}$ and 
$f(v)\equiv{v}$,
or $f(u)\equiv{v}$ and $f(v)\equiv{u}$ {\it mod} $\sqrt\sigma$.
Let $P=f|_{\sqrt\sigma}$, and suppose that $f(t)\equiv{t}^\epsilon$
{\it mod} $\sqrt\sigma$ .
Then $f_*[\sigma]=\epsilon\det(P)[\sigma]$.

In the first case, $f(t)\equiv{t}$ {\it mod} $\sqrt\sigma$,
while $Pc_u|_{\sqrt\sigma}=c_u|_{\sqrt\sigma}P$ and $Pc_v|_{\sqrt\sigma}=c_v|_{\sqrt\sigma}P$, 
and so $P=I$. 
In the second case, $f(t)\equiv{t^{-1}}$ {\it mod} $\sqrt\sigma$,
while $Pc_u|_{\sqrt\sigma}P^{-1}=c_v|_{\sqrt\sigma}$ 
and $Pc_v|_{\sqrt\sigma}P^{-1}=c_u|_{\sqrt\sigma}$.
Hence $P^2=I$.
Since $c_t=c_uc_v$ and $c_t|_{\sqrt\sigma}$ has infinite order, 
$P\not=\pm{I}$.
Therefore $\det{P}=-1$. 
In each case, $f$ is orientation preserving.
\end{proof}

If $\sigma$ is a $\mathbb{S}ol^3$-group then $Out(\sigma)$ is finite,
by Theorem 8.10 of \cite{Hi}.

\begin{theorem}
Suppose that $B=\mathbb{A}$ or $\mathbb{M}b$.
Then $M$ is orientable 
if and only if $\pi\cong\sigma\rtimes\mathbb{Z}$, 
where $\sigma$ is a $\mathbb{S}ol^3$-group such that 
$\sigma/\sqrt\sigma\cong{D_\infty}$.
If $M$ is non-orientable then $\pi\cong{B_1}\rtimes\mathbb{Z}$.
\end{theorem}

\begin{proof}
If $M$ is Seifert fibred over $B=\mathbb{A}$ or $\mathbb{M}b$ then
$\beta_1(M)=\beta_1(\beta)=1$.
Hence there is an unique splitting 
$\pi=\kappa\rtimes_\theta\mathbb{Z}$.
Moreover, $N<\kappa$ and $\kappa/N\cong{D_\infty}$,
since $B=\mathbb{A}$ or $\mathbb{M}b$.  
If $\kappa$ is a $\mathbb{S}ol^3$-group then $N=\sqrt\kappa$,
since $\sqrt\kappa$ is characteristic and $\nu/\sqrt\kappa$ 
has no nontrivial finite normal subgroup.
Since $\kappa$ is orientable and $\theta$ is orientation preserving,
by Lemma 6, $M$ is orientable.

Conversely, if $\pi\cong\sigma\rtimes\mathbb{Z}$, 
where $\sigma$ is a $\mathbb{S}ol^3$-group such that 
$\sigma/\sqrt\sigma\cong{D_\infty}$, then $N=\sqrt\sigma$
and so $\pi/N\cong{D_\infty}\rtimes\mathbb{Z}$.
Hence $B=\mathbb{A}$ or $\mathbb{M}b$.

If $\kappa$ maps onto $D_\infty$ and is virtually abelian then $[\kappa:\sqrt\pi]=2$,
by Theorem 8.4 of \cite{Hi}.
Since $\kappa$ is not $G_2$, by Lemma 5, it must be $B_1$ or $B_2$,
and since $B_2$ does not map onto $D_\infty$,
we must have $\kappa=B_1$.
Hence $N=\zeta{B_1}\cong\mathbb{Z}^2$,
since $\zeta{D_\infty}=1$ and $B_1/\zeta{B_1}\cong{D_\infty}$.
Since $B_1$ is non-orientable, $M$ is non-orientable.
\end{proof}

There are examples of each type allowed by Theorem 7.
For instance, let $\sigma$ be the $\mathbb{S}ol^3$-group with presentation
\[
\langle{x,y,u,v}\mid{xy=yx,}~u^2=x,~uyu^{-1}=y,~v^2=x^3y^{-2}, vxv^{-1}=x^{17}y^{-12},
\]
\[
vyv^{-1}=x^{24}y^{-17}\rangle.
\]
Then $\sqrt\sigma=\langle{x,y}\rangle$ and $\sigma/\sqrt\sigma\cong{D_\infty}$.
We may define an involution $f$ of $\sigma$ by $f(u)=v$, $f(y)=x^4y^{-3}$ and $f(v)=u$.
The groups $\sigma\times\mathbb{Z}$ and $\sigma\rtimes_f\mathbb{Z}$ 
are groups of orientable $\mathbb{S}ol^3\times\mathbb{E}^1$-manifolds
which are Seifert fibred over $\mathbb{A}$ and $\mathbb{M}b$, respectively.

The flat 3-manifold group $B_1$ has a presentation 
\[
\langle{X,y,z}\mid{Xz=zX},~XyX^{-1}=y^{-1},~yz=zy\rangle.
\]
Let $\theta$ and $\psi$ be the automorphisms defined by
$\theta(X)=X^3z^4$, $\theta(y)=y$ and $\theta(z)=X^2z^3$,
and $\psi(X)=Xyz$, $\psi(y)=y^{-1}$ and $\psi(z)=X^2z^3$, 
respectively.
Then $B_1\rtimes_\theta\mathbb{Z}$ and $B_1\rtimes_\psi\mathbb{Z}$
are the groups of non-orientable $\mathbb{S}ol^3\times\mathbb{E}^1$-manifolds
which are Seifert fibred over $\mathbb{A}$ and $\mathbb{M}b$,
respectively.

\section{examples with $\pi/\pi'$ finite}

Suppose now that $\beta_1(\pi)=0$.
Then $\pi/\nu\cong{D_\infty}$, and so $\pi\cong{G*_\nu{H}}$, 
where $\nu$ is the preimage of the maximal finite normal subgroup of
$\mathrm{Im}(\alpha)$ and $[G:\nu]=[H:\nu]=2$.
Moreover, either $\nu=\sqrt\pi$ and $\mathrm{Im}(\alpha)\cong{D_\infty}$
or $\nu\cong{G_2}\cong\mathbb{Z}_{-I}\mathbb{Z}$ and $\mathrm{Im}(\alpha)\cong\widetilde{D}_\infty$ or 
$D_\infty\times{Z/2Z}$.
Since $\nu/N$ is a normal subgroup of $\beta$ it has 
no nontrivial finite normal subgroup.
Therefore if $\nu\cong\mathbb{Z}^3$ then $\nu/N\cong\mathbb{Z}$
and $N$ is a direct summand of $\nu$,
while if $\nu\cong{G_2}$ then either $\nu/N\cong\mathbb{Z}$ and $N=I(\nu)$
or $\nu/N\cong{D_\infty}$.
(However $N$ need not be a canonical subgroup of $\nu$.)

In order to describe our examples clearly, 
we should be more precise about our definitions of
such amalgamated free products.
We shall assume that $\nu$ is given as a subgroup of $G$ 
and that $\phi:\nu\to{H}$ is a monomorphism.
Then we shall write
\[
G*_\phi{H}=G*H/\langle\langle\phi(n)=n,~\forall~n\in\nu\rangle\rangle.
\]

Since $\nu$, $G$ and $H$ are each finite extensions of $\sqrt\pi$,
they are flat 3-manifold groups.
If $G$ and $H$ were both non-orientable then $\pi$ would be orientable, 
and so $\beta_1(\pi)=1+\frac12(\beta_2(\pi)-\chi(\pi))>0$,
contrary to the assumption.
Hence we may assume that $H$ is orientable.
A Mayer-Vietoris argument shows that $\beta_1(G)+\beta_1(H)\leq\beta_1(\nu)$.

If $\nu=\sqrt\pi$ then $G$ and $H$ each have holonomy of order $\leq2$,
and so $\beta(G)$ and $\beta(H)$ are each $>0$.
We may then assume that $H=G_2$ and $G=G_2,B_1$ or $B_2$.
In each case $\nu=\sqrt{G}=\sqrt{H}$.

If $\nu\cong{G_2}$ then we may assume that $H\cong{G_6}$ and
$G\cong{G_2,G_4,G_6}$, $B_3$ or $B_4$.
If $G\cong{G_4},B_3$ or $B_4$ it has an unique subgroup of index 2
which is isomorphic to $G_2$, 
while if $G\cong{G_2}$ or $G_6$ there are three such subgroups,
which are equivalent under automorphisms of $G$.

We shall use the following presentations for these groups:
\[
\mathbb{Z}^3=\langle{x,y,z}\mid{xy=yx},~xz=zx,~yz=zy\rangle,
\]
\[
G_2=\langle{r,y,z}\mid{ryr^{-1}=y^{-1}},~rzr^{-1}=z^{-1},~yz=zy\rangle,
\]
\[
G_4=\langle{t,y,z}\mid{tyt^{-1}=z},~tzt^{-1}=y^{-1},~yz=zy\rangle,
\]
\[
G_6=\langle{r,s}\mid{rs^2r^{-1}=s^{-2}},~sr^2s^{-1}=r^{-2}\rangle,
\]
\[
B_1=\langle{X,y,z}\mid{Xz=zX},~XyX^{-1}=y^{-1},~yz=zy\rangle,
\]
\[
B_2=\langle{X,y,z}\mid{XyX^{-1}=y^{-1}},~XzX^{-1}=yz,~yz=zy\rangle,
\]
\[
B_3=\langle t,Y,z\mid{tYt^{-1}=Y^{-1}},~tz=zt,~ YzY^{-1}=z^{-1}\rangle,
\]
and
\[
B_4=\langle t,Y,z\mid{tYt^{-1}=Y^{-1}z},~ tz=zt,~YzY^{-1}=z^{-1}\rangle.
\]
(Here $r^2\in{G_2}$, $t^4\in{G_4}$, $r^2\in{G_6}$,
$X^2\in{B_1}$, $X^2\in{B_2}$, $t^2\in{B_3}$ and $t^2\in{B_4}$
correspond to $x\in\mathbb{Z}^3$, while $s^2\in{G_6}$, 
$Y^2\in{B_3}$ and $Y^2\in{B_4}$ correspond to $y\in\mathbb{Z}^3$, 
and $(rs)^2\in{G_6}$ corresponds to $z\in\mathbb{Z}^3$.)
In each case, let $A(G)$ be the maximal abelian normal subgroup
of the flat 3-manifold group $G$.

In order to realize the remaining bases $S(2,2,2,2)$,
$P(2,2)$ and $\mathbb{D}(2,2)$, 
it shall suffice to consider the case $\nu\cong\mathbb{Z}^3$.
We shall assume that $A(\nu)$, 
$A(G)$ and $A(H)$ have bases $\{x,y,z\}$, as above.
Clearly $N<A(\nu)\leq{A(G)}\cap{A(H)}$.
Since $G/N$ has $\nu/N\cong\mathbb{Z}$ as an index-2 subgroup,
$G/N\cong\mathbb{Z}$, $\mathbb{Z}\oplus{Z/2Z}$ or $D_\infty$.
If $G\cong{G_2}$ then either $N=I(G)=\langle{y,z}\rangle$ 
or $N\cong\langle{x,w}\rangle$ for some $w\in{I(G)}$,
since $N$ is normal in $G$.
However if $N=I(G)$ then $G$ would act on $N$ via $-I$, and so $\mathrm{Im}(\alpha)$ would be finite.
Hence we may assume that $N\cong\langle{x,y}\rangle$,
and so $G/N\cong{D_\infty}$.

If $G\cong{B_1}$ then either $N\cong\langle{x^az^b,y}\rangle$,
with $(2a,b)=1$, and $G/N\cong\mathbb{Z}$,
or $N\cong\langle{x^az^{2c},y}\rangle$,
with $(2a,c)=1$, and $G/N\cong\mathbb{Z}\oplus{Z/2Z}$,
or $N\cong\langle{x,z}\rangle$, and $G/N\cong{D_\infty}$.
But in the latter case $G$ would act on $N$ via $-I$, and so $\mathrm{Im}(\alpha)$ would be finite.

Similarly, if $B_2$ we find that $G/N$ can be either $\mathbb{Z}$
or $\mathbb{Z}\oplus{Z/2Z}$.
However, $B_2$ does not admit any epimorphisms to $D_\infty$.

Let $G,H\cong{G_2}$, and let $\phi:A(G)\to{A(H)}$ be the isomorphism
given in terms of standard bases 
by the bordered $3\times3$ matrix
$C\oplus[1]=\left(\begin{smallmatrix}
C&0\\ 
0&1
\end{smallmatrix}\right)$,
where $C=\left(\begin{smallmatrix}
1&1\\ 
1&2
\end{smallmatrix}\right)$,
and let $\pi=G*_\phi{H}$.
Then $N=\langle{x,y}\rangle<G$ is normal in $\pi$, 
and $\beta=\pi/N\cong{D_\infty*_\mathbb{Z}D_\infty}$.
Let $u\in{G}$ and $v\in{H}$ correspond to $r\in{G_2}$.
Then the action of $uv$ on $\nu$ by conjugation has matrix 
$\left(\begin{matrix}
3&-2&0\\ 
-4&3&0\\
0&0&1
\end{matrix}\right)$.
The corresponding semidirect product is the fundamental group 
of a mapping torus which is a $\mathbb{S}ol^3\times{E}^1$-manifold.
Hence the overgroup $\pi$ is the fundamental group of 
a $\mathbb{S}ol^3\times{E}^1$-manifold 
which is Seifert fibred over $B=S(2,2,2,2)$.
If we set $G=B_1$ instead, and use the same matrices,
we get an example with $B=\mathbb{D}(2,2)$ instead,
since $\pi/N\cong(\mathbb{Z}\oplus{Z/2Z})*_\mathbb{Z}D_\infty$.
Modifying the $3\times3$ matrix so that the entries of its third column 
are all 1s gives an example with $B=P(2,2)$,
since $\pi/N\cong\mathbb{Z}*_\mathbb{Z}D_\infty$.

Similar examples can be constructed when $\nu\cong{G_2}$.
We then have $G/N\cong\mathbb{Z}$ (and $N=I(G)$) if $G\cong{G_4}$,
$G/N\cong\mathbb{Z}\oplus{Z/2Z}$ if $G=B_3$ or $B_4$, 
and $G/N\cong{D_\infty}$ if $G\cong{G_6}$.
Restriction from $G_2$ to $I(G_2)$ induces an epimorphism from
$Aut(G_2)$ to $Aut(I(G_2))$.
Thus given $C\in{GL(2,\mathbb{Z})}$ and $G=G_4, G_6, B_3$ or $B_4$
there is an embedding of $G_2$ in $G$ whose restriction to $I(G_2)$
has matrix $C$ with respect to the standard bases, and which fixes $t$.
As before, the corresponding groups $\pi=G*_\phi{H}$ are
$\mathbb{S}ol^3\times{E}^1$-groups, 
with $\beta=\pi/N\cong\mathbb{Z}*_\mathbb{Z}D_\infty$,
$(\mathbb{Z}\oplus{Z/2Z})*_\mathbb{Z}D_\infty$ 
or ${D_\infty}*_\mathbb{Z}D_\infty$, respectively.

It is easy to see that every flat 3-manifold group or 
$\mathbb{S}ol^3$-group can be generated by at most three elements,
and hence that every $\mathbb{S}ol^3\times\mathbb{E}^1$-group
requires at most four generators.
This is best possible in general.
If $A=\left(\begin{smallmatrix}
1&2\\ 
2&5
\end{smallmatrix}\right)$ 
then $\sigma=\mathbb{Z}^2\rtimes_A\mathbb{Z}$ is a $\mathbb{S}ol^3$-group 
such that $H_1(\sigma;\mathbb{F}_2)\cong\mathbb{F}_2^3$,
and so $\sigma\times\mathbb{Z}$ is a
$\mathbb{S}ol^3\times\mathbb{E}^1$-group 
that cannot be generated by three elements.
Similarly, if 
$\phi=\left(\begin{smallmatrix}
A&0\\ 
0&1
\end{smallmatrix}\right)\in{GL(3,\mathbb{Z})}$ then $\pi=G_2*_\phi{G_2}$ needs four generators.

\section{outline of the classification in terms of Seifert data}

The subgroup $N$ is characteristic in $\pi$.
Therefore any isomomorphism $f:\pi\to\tilde\pi$ of such groups 
induces isomorphisms $f|_N:N\to\tilde{N}$ and $\bar{f}:\pi/N\to\tilde\pi/\tilde{N}$.
Hence the classification of such groups may be derived from the classification 
of the extensions 
\[
\xi:\quad 1\to{N}\to\pi(\xi)\to\beta\to1.
\]
The ingredients of such a classification are the quotient group $\beta$, 
the action $\alpha:\beta\to{Out(N)}\cong{GL(2,\mathbb{Z})}$ and 
the cohomology class $e(\xi)\in{H^2}(\beta;N^\alpha)$,
where $N^\alpha$ is $N$ considered as a $\mathbb{Z}[\beta]$-module with
module structure determined by $\alpha$.
Given $\beta$, $N$ and $\alpha$, the groups $\pi(\xi)$ and $\pi(\xi')$ are isomorphic 
if and only if $e(\xi')=g_\#e(\xi)$, where $g$ is a $\beta$-linear automorphism of $N$.

The group $\pi(\xi)$ determined by such an extension is the 
fundamental group of a $\mathbb{S}ol^3\times\mathbb{E}^1$-manifold 
if and only if 
it is torsion-free and $\mathrm{Im}(\alpha)$ contains a matrix with trace $>2$.
The torsion condition can be checked by restricting the extension 
to the finite cyclic subgroups of $\beta$. 
In all cases of interest to us, $\beta$ is either torsion-free or is a semidirect product 
$\gamma\rtimes{Z/2Z}$ where $\gamma$ is torsion-free.
The 2-torsion must act via $I$ or $\pm{U}$,
where
$U=\left(\begin{smallmatrix}
1&0\\ 
0&-1
\end{smallmatrix}\right)$,
(and not via $-I$ or 
$\pm\left(\begin{smallmatrix}
0&1\\ 
1&0
\end{smallmatrix}\right)$).

\begin{lemma}
If $\beta\cong\gamma\rtimes{Z/2Z}$ with $\gamma$ torsion-free then $\pi(\xi)$
is torsion-free if and only if $e(\xi|_{Z/2Z})\not=0$ in $H^2(Z/2Z;(\mathbb{Z}^2)^\alpha)$.
\end{lemma}

\begin{proof}
Since the non-trivial finite subgroups of $\beta$ have order 2,
and are all conjugate, $\pi$ is torsion-free if and only if  
any one of these subgroups has torsion-free preimage in $\pi$.
This is an extension of $Z/2Z$ by $\mathbb{Z}^2$,
with action trivial or via $U$,
and the claim follows easily.
\end{proof}

The identification of $N$ with $\mathbb{Z}^2$ and $\pi/N$ with $\beta$ are only 
well-defined up to compositions with automorphisms, and so the same is true for the action 
$\alpha$. 

Let $D(P)$ be the subgroup of $GL(2,\mathbb{Z})$ generated by 
$U$ and $V=PUP^{-1}$, where $P\in{GL(2,\mathbb{Z})}$ 
is such that $UV$ has infinite order.
Then $D(P)\cong{D_\infty}$.
If $\mathrm{Im}(\alpha)\cong{D_\infty}$ then it is generated by elements conjugate to $U$, and so is conjugate to some $D(P)$.
There is a matrix with trace $>2$ in $\mathrm{Im}(\alpha)$
if and only if $UV$ has an eigenvalue $\not=\pm1$. 
The pair of involutions $u,v$ generating $D_\infty$ 
is unique up to interchange and (simultaneous) conjugation.
Therefore $D(P)$ is conjugate to $D(\widetilde{P})$ if and only if 
$\widetilde{P}=U^\delta{P^\epsilon}U^\delta$,
where $\delta=0$ or 1 and $\epsilon=\pm1$.

We may find the epimorphisms from $\beta$ to $D_\infty$
by considering the possible kernels, 
which are normal subgroups of Hirsch length 1.

If $B=S(2,2,2,2)$ then $\beta=\mathbb{Z}^2\rtimes_{-1}{Z/2Z}$,
so every subgroup of $\sqrt\beta=\mathbb{Z}^2$ is normal, 
while normal subgroups with non-trivial torsion have finite index.
In this case the normal subgroups of Hirsch length 1 are infinite cyclic,
and all epimorphisms from $\beta$ to $D_\infty$ 
are equivalent up to composition with an automorphism of $\beta$.
Similarly, all epimorphisms to $D_\infty\times{Z/2Z}$ are equivalent.
On the other hand, this group has no epimorphisms to $\widetilde{D}_\infty$.

If $B=\mathbb{D}(2,2)$ then
$\beta=\langle{j,u,v}\mid{ujv=jvu,~j^2=u^2=v^2=1}\rangle$,
and there are just two maximal normal subgroups of Hirsch length 1, 
namely $\langle{jv}\rangle$ and $\langle{u,(ju)^2}\rangle$.
The quotients by $\langle(jv)^2\rangle$ and $\langle(ju)^2\rangle$ 
are each $D_\infty\times{Z/2Z}$.
In each case, the epimorphisms are inequivalent.
On the other hand, this group has no epimorphisms to $\widetilde{D}_\infty$.

If $B=P(2,2)$ then
$\beta=\langle{s,u}\mid{(us^2)^2=u^2=1}\rangle$,
and there are again just two maximal normal subgroups of Hirsch length 1, 
namely $\langle{s^2}\rangle$ and $\langle{(us)^2}\rangle$.
There is an automorphism that fixes $u$ and swaps $s$ with $us$.
The quotients by $\langle{s^4}\rangle$ and $\langle(us)^4\rangle$
are each $\widetilde{D}_\infty$.
In each case, the epimorphisms are equivalent.
On the other hand, since this group has a 2 generator presentation
it has no epimorphism to $D_\infty\times{Z/2Z}$.

The possible actions $\alpha$ with $\mathrm{Im}(\alpha)\cong{D_\infty}$
and with given kernel may be parametrized by 
matrices $P\in{GL(2,\mathbb{Z})}$ such that
$UPUP^{-1}$ has an eigenvalue $\not=\pm1$, 
modulo inversion and conjugation by $U$.
Similarly for epimorphisms to $D_\infty\times{Z/2Z}$, 
since the $Z/2Z$ direct factor must be generated by $\pm{I}$.
Let $W=\pm\left(\begin{smallmatrix}
0&1\\ 
-1&0
\end{smallmatrix}\right)$.
Then actions with $\mathrm{Im}(\alpha)\cong\widetilde{D}_\infty$
are conjugate to actions generated by $U$ and $PWP^{-1}$,
such that $UPWP^{-1}$ has an eigenvalue $\not=\pm1$, 
modulo conjugation by $U$.

The final stage of the classification is the determination 
of the possible extensions with given base and action, 
modulo automorphisms of the coefficients.
We shall settle for a slightly weaker result.

\begin{theorem}
There are only finitely many $\mathbb{S}ol^3\times\mathbb{E}^1$-groups 
with given base group $\beta$ and action $\alpha$.
\end{theorem}

\begin{proof}
The cohomology groups $H^2(\beta;N^\alpha)$ may be estimated 
by using the LHS spectral sequence
\[
H^p(\beta/\sqrt\beta;H^q(\sqrt\beta;N^\alpha))\Rightarrow
{H^{p+q}(\beta;N^\alpha)}
\] 
for $\beta$ as a extension of the finite group $\beta/\sqrt\beta$ by $\sqrt\beta\cong\mathbb{Z}^2$.
Now $H^0(\sqrt\beta;N^\alpha)=0$, 
since $\alpha(\sqrt\beta)$ contains matrices with no eigenvalue $1$,
and  $H^q(\sqrt\beta;N^\alpha)=0$, for $q>2$.
Hence this spectral sequence has just two nonzero columns,
and so there is an exact sequence
\[
H^1(\beta/\sqrt\beta;H^1(\sqrt\beta;N^\alpha))\to
{H^2(\beta;N^\alpha)}\to{H^0(\beta/\sqrt\beta;H^2(\sqrt\beta;N^\alpha))}.
\]
Since $H^q(\sqrt\beta;N^\alpha)$ is finitely generated,
for all $q$, the first term is finite.
Poincar\'e duality for $\sqrt\beta$ gives
$H^2(\sqrt\beta;N^\alpha)\cong{H_0(\sqrt\beta;{N^\alpha})}$.
This is again finite, 
since $\alpha(\sqrt\beta)$ contains matrices with no eigenvalue $1$,
and so the result follows.
\end{proof}

\end{document}